\documentclass[a4paper,reqno,oneside]{amsart}

\usepackage{setspace}
\usepackage[totalwidth=13cm,totalheight=22cm]{geometry}
\usepackage[T1]{fontenc}
\usepackage[dvips]{graphicx}
\usepackage{epsfig}
\usepackage{amsmath,amssymb,amscd,amsthm}
\usepackage{subfigure}
\usepackage[utf8]{inputenc}
\usepackage{latexsym}
\usepackage{graphics}
\usepackage{colortbl}
\usepackage{color}
\usepackage{ae}
\usepackage{enumitem}
\usepackage[all]{xy}
\usepackage{scalerel}
\usepackage{tikz}
\usepackage{cancel}

\usepackage{arcs}

\usepackage{bbm}

\makeatletter
\newtheorem*{rep@theorem}{\rep@title}
\newcommand{\newreptheorem}[2]{%
\newenvironment{rep#1}[1]{%
 \def\rep@title{#2 \ref{##1}}%
 \begin{rep@theorem}}%
 {\end{rep@theorem}}}
\makeatother

\makeatletter
\newtheorem*{rep@cor}{\rep@title}
\newcommand{\newrepcor}[2]{%
\newenvironment{rep#1}[1]{%
 \def\rep@title{#2 \ref{##1}}%
 \begin{rep@cor}}%
 {\end{rep@cor}}}
\makeatother

\makeatletter
\newtheorem*{rep@prop}{\rep@title}
\newcommand{\newrepprop}[2]{%
\newenvironment{rep#1}[1]{%
 \def\rep@title{#2 \ref{##1}}%
 \begin{rep@prop}}%
 {\end{rep@prop}}}
\makeatother

\newtheorem{cor}{Corollary}[section]
\newtheorem{corx}{Corollary}

\newtheorem{theorem}[cor]{Theorem}
\newtheorem{thmx}[corx]{Theorem}

\newtheorem{prop}[cor]{Proposition}

\newrepcor{cor}{Corollary}
\newreptheorem{theorem}{Theorem}
\newrepprop{prop}{Proposition}

\newtheorem{lemma}[cor]{Lemma}

\theoremstyle{definition}
\newtheorem{defi}[cor]{Definition}
\theoremstyle{remark}
\newtheorem{remark}[cor]{Remark}
\newtheorem*{remark*}{Remark}

\newcommand{\C}{{\mathbb C}}

\newcommand{\R}{{\mathbb R}}

\newcommand{\Hyp}{\mathbb{H}}

\newcommand{\PSL}{\mathrm{PSL}}

\newcommand{\ph}{\varphi}

\newcommand{\RP}{\R\mathrm{P}}
\newcommand{\f}{\mathfrak{f}}
\newcommand{\g}{\mathfrak{g}}

\newcommand{\tr}{\mathrm{tr}\,}

\newcommand{\grad}{\operatorname{grad}}

\newcommand{\D}{\mathbb{D}}

\def\Hess{\mathrm{Hess}}

\begin{document}

\setcounter{secnumdepth}{3}
\setcounter{tocdepth}{2}

\title[On the maximal dilatation of minimal Lagrangian extensions]{On the maximal dilatation of quasiconformal minimal Lagrangian extensions}

\author[Andrea Seppi]{Andrea Seppi}
\address{Andrea Seppi: CNRS and Université Grenoble Alpes, 100 Rue des Mathématiques 38610 Gières, France.} \email{andrea.seppi@univ-grenoble-alpes.fr}



\begin{abstract}
Given a quasisymmetric homeomorphism $\ph$ of the circle, Bonsante and Schlenker proved the existence and uniqueness of the minimal Lagrangian extension $f_\ph:\Hyp^2\to\Hyp^2$ to the hyperbolic plane. By previous work of the author, its maximal dilatation satisfies $\log K(f_\ph)\leq C||\ph||_{cr}$, where $||\ph||_{cr}$ denotes the cross-ratio norm. We give constraints on the value of an optimal such constant $C$, and discuss possible lower inequalities, by studying two one-parameter families of minimal Lagrangian extensions in terms of maximal dilatation and cross-ratio norm.
\end{abstract}

\maketitle


\section{Introduction}

A classical problem in quasiconformal Teichm\"uller theory consists in finding particular quasiconformal extensions $f:\D\to\D$, where $\D$ is the unit disc in $\C$, of a quasisymmetric homeomorphism $\ph:\mathbb \partial\D\to\mathbb \partial\D$ of the circle. Classical extensions are for instance the \emph{Beurling-Ahlfors} extension (\cite{ahl_beur,lehtinen}) and the \emph{Douady-Earle} extension (\cite{douadyearle}). More recently  Markovic (\cite{mar_schoenconj}), and independently Benoist and Hulin (\cite{benoisthulin}), proved the existence of the \emph{harmonic} extension, hence solving the so-called \emph{Schoen conjecture}. As the notion of harmonicity depends on the choice of a Riemannian metric on the target, the Poincaré metric is considered on $\D$, thus identifying the target $\D$ with the hyperbolic plane $\Hyp^2$. Uniqueness of the harmonic extension was previously known (see \cite{litamuniqueness}).

\subsection*{Minimal Lagrangian extensions}

In \cite{bon_schl}, Bonsante and Schlenker proved the existence and uniqueness of the \emph{minimal Lagrangian} extension of $\ph$ --- where the notion of minimal Lagrangian diffeomorphism depends on the choice of the Riemannian metric (here, the Poincaré metric) both on the source and on the target.
Indeed, an orientation-preserving diffeomorphism $f:\Hyp^2\to\Hyp^2$ is called minimal Lagrangian if it is area-preserving and its graph is a minimal surface in the Riemannian product $\Hyp^2\times\Hyp^2$. Minimal Lagrangian diffeomorphisms $f:\Hyp^2\to\Hyp^2$ are also closely related to harmonic maps, as they are characterized by the condition of being the composition $f=h_1\circ (h_2)^{-1}$, where $h_1,h_2$ are harmonic diffeomorphisms to $\Hyp^2$ with opposite Hopf differential -- see for instance \cite{labourieCP,bon_schl}. This paper concerns the study of the minimal Lagrangian extension of a quasisymmetric homeomorphism $\ph:\partial\Hyp^2\to\partial\Hyp^2$, which we will simply denote by $f_\ph$.

More concretely, let us consider $\Hyp^2$ in the upper half-plane model, so that its boundary is identified to $\RP^1\cong \R\cup\{\infty\}$. Given a quasisymmetric homeomorphism $\ph:\mathbb \RP^1\to\mathbb \RP^1$, its  \emph{cross-ratio norm} is defined as:
$$||\ph||_{cr}:=\sup_{cr(Q)=-1}\left|\log\left|cr(\ph(Q))\right|\right|~,$$
where $cr(Q)$ denotes the cross-ratio of a quadruple of points $Q$ in $\RP^1$, and our definition of cross-ratio is such that a quadruple $Q$ is symmetric if and only if $cr(Q)=-1$.

Then, if $K(f_\ph)$ denotes the maximal dilatation of the minimal Lagrangian extension $f_\ph:\Hyp^2\to\Hyp^2$ of $\ph$, the main result of \cite{seppimaximal} is the following inequality:
\begin{equation} \label{eq inequality intro}
\log K(f_\ph)\leq C||\ph||_{cr}~,
\end{equation}
where $C>0$ is a universal constant. We remark that similar results were obtained for the Beurling-Ahlfors extension:
$$\log K(f_\ph^{B\!A})\leq 2||\ph||_{cr}\,,$$
as proved in \cite{ahl_beur}, while \cite{lehtinen} showed:
$$\log K(f_\ph^{B\!A})\leq ||\ph||_{cr}+\log 2\,.$$
Concerning the Douady-Earle extension, in \cite{hu_muzician_douadyearle} Hu and Muzician proved the inequality:
$$\log K(f_\ph^{D\!E})\leq C||\ph||_{cr}\,,$$
thus improving previous results of \cite{douadyearle}.

It seems therefore a natural question to ask what is the best possible value of the constants, for the different extensions taken into account.
Unfortunately, the methods used in \cite{seppimaximal} do not give any control on the value of the constant $C$ appearing in \eqref{eq inequality intro}. On the other hand, a partial lower bound on the constant $C$ is provided by the following estimate, again proved in \cite{seppimaximal}: for every $\epsilon>0$ there exists $\delta>0$ such that:
\begin{equation} \label{eq inequality intro2}
\log K(f_\ph)\geq \left(\frac{1}{2}-\epsilon\right)||\ph||_{cr}~,
\end{equation}
provided $||\ph||_{cr}\leq \delta$.

\subsection*{Two families of minimal Lagrangian diffeomorphisms}

In this paper, we give a rather explicit description of two families of minimal Lagrangian diffeomorphisms, each family depending on a real parameter, for which $||\ph||_{cr}$ varies between $0$ and $+\infty$. We therefore study the best possible value of the constant $C$ for each of these families, as $||\ph||_{cr}$ approaches $0$ and $+\infty$. This will provide more precise conditions on the limit inferior and limit superior of the ratio ${\log K(f_\ph)}/{||\ph||_{cr}}$ as $||\ph||_{cr}\to 0$ and $||\ph||_{cr}\to+\infty$.

The first family of examples consists in the extensions $f_{\ph_\lambda}$ of a \emph{simple (left) earthquake map} on $\partial\Hyp^2$. That is, we consider for every $\lambda\geq 0$ the orientation-preserving homeomorphisms:
$$\ph_\lambda(x)=\begin{cases} x & \textrm{if }x\leq 0 \\ e^{\lambda} x & \textrm{if }x> 0 \end{cases}~.$$
It turns out that $\ph_\lambda$ is quasisymmetric and $||\ph_\lambda||_{cr}=\lambda$. Observe that $\ph_0$ is the identity. Then we prove the following:
 
\begin{thmx} \label{thm simple earth k=0} \label{thm simple earth k=1}
Let $f_{\ph_\lambda}:\Hyp^2\to\Hyp^2$ be the minimal Lagrangian diffeomorphisms which extend a simple left earthquake $\ph_\lambda$ of weight $\lambda$. Then
$$\lim_{\lambda\to 0^+}\frac{\log K(f_{\ph_\lambda})}{||\ph_\lambda||_{cr}}=\frac{2}{\pi}\qquad\text{and}\qquad \lim_{\lambda\to +\infty}\frac{\log K(f_{\ph_\lambda})}{||\ph_\lambda||_{cr}}=\frac{\sqrt 2}{2}~.$$
\end{thmx}
\noindent Numerically, we have
$$\frac{2}{\pi}\approx 0.64~, \qquad\frac{\sqrt 2}{2}\approx 0.71~.$$

We remark that the Douady-Earle extension $f^{D\!E}_{\ph_\lambda}$ of simple earthquakes was studied in \cite{hu_muzician_douadyearle}, where the following inequality was obtained:
$$\lim_{\lambda\to +\infty}\frac{\log K(f^{D\!E}_{\ph_\lambda})}{||\ph_\lambda||_{cr}}\geq \frac{1}{4}~.$$

The second family considered consists in the extensions  of \emph{power functions} $\psi_\alpha:\RP^1\to\RP^1$. Namely, we consider the orientation-preserving homeomorphisms:
$$\psi_\alpha(x)=\begin{cases} x^\alpha & \textrm{if }x\geq 0 \\ -|x|^\alpha  & \textrm{if }x< 0 \end{cases}~.$$
When $\alpha=1$, $\psi_1$ is the identity. Since $(\psi_\alpha)^{-1}=\psi^{1/\alpha}$, and the inverse of a minimal Lagrangian map is minimal Lagrangian with the same maximal dilatation, we will only consider the case $\alpha\in[1,+\infty).$
We will prove the following:

\begin{thmx} \label{thm power a=0}
Let $f_{\psi_\alpha}:\Hyp^2\to\Hyp^2$ be the minimal Lagrangian diffeomorphisms which extend the power function $\psi_\alpha$. Then
$$\limsup_{\alpha\to 1^+}\frac{\log K(f_{\psi_\alpha})}{||\psi_\alpha||_{cr}}\leq  \frac{\sqrt 2}{\log \left(3+\sqrt 2\right)}\qquad\text{and}\qquad \lim_{\alpha\to +\infty}\frac{\log K(f_{\psi_\alpha})}{||\psi_\alpha||_{cr}}=0~.$$
\end{thmx}
\noindent Numerically, the first upper bound is ${\sqrt 2}/{\log \left(3+\sqrt 2\right)}\approx 0.80.$

Since Theorem \ref{thm simple earth k=0} exhibits a family of minimal Lagrangian extensions for which the ratio ${\log K(f_\ph)}/{||\ph||_{cr}}$ converges to a positive value as $||\ph||_{cr}\to+\infty$, we obtain as a consequence that the inequality \eqref{eq inequality intro} of \cite{seppimaximal} is \emph{sharp};

\begin{corx}
There exists no function $F:[0,+\infty)\to[0,+\infty)$ which grows less than linearly such that ${\log K(f_\ph)}\leq F({||\ph||_{cr}})$ for every minimal Lagrangian extension $f_\ph:\Hyp^2\to\Hyp^2$ of a quasisymmetric homeomorphism $\ph: \RP^1\to\RP^1$.
\end{corx}

On the other hand, Theorem \ref{thm power a=0} shows an inequality of the form \eqref{eq inequality intro2}, as obtained in \cite{seppimaximal}, cannot hold globally, but only holds for bounded values of $||\ph||_{cr}$.

\begin{corx}
There exists no lower bound of the form ${\log K(f_\ph)}\geq C{||\ph||_{cr}}$, that holds for every minimal Lagrangian extension $f_\ph:\Hyp^2\to\Hyp^2$ of a quasisymmetric homeomorphism $\ph: \RP^1\to\RP^1$.
\end{corx}

\subsection*{Applications to the study of optimal constants}

As already mentioned, Theorem \ref{thm simple earth k=0}  and Theorem \ref{thm power a=0}, as an application, provide explicit bounds on the quantities:
$$\liminf\frac{\log K(f_\ph)}{||\ph||_{cr}}\qquad\text{and}\qquad \limsup\frac{\log K(f_\ph)}{||\ph||_{cr}}~,$$
as $||\ph||_{cr}\to 0$ and $||\ph||_{cr}\to +\infty$, which can be compared with the known estimates for the more classical extensions. More precisely, we obtain:

\begin{corx} \label{cor ph to 0}
If $f_\ph:\Hyp^2\to\Hyp^2$ denotes the minimal Lagrangian extension of a quasisymmetric homeomorphism $\ph: \RP^1\to\RP^1$, then:
$$\liminf_{||\ph||_{cr}\to 0^+}\frac{\log K(f_\ph)}{||\ph||_{cr}}\in\left[\frac{1}{2},\frac{2}{\pi}\right]\qquad\text{and}\qquad \limsup_{||\ph||_{cr}\to 0^+}\frac{\log K(f_\ph)}{||\ph||_{cr}}\in\left[\frac{2}{\pi},+\infty\right)~.$$
\end{corx}

Such result can be compared for instance with the Beurling-Ahlfors extension $f_\ph^{B\!A}$, for which the limit superior of $\log K(f_\ph^{B\!A})/||\ph||_{cr}$ is at most $2$ as $||\ph||_{cr}\to 0^+$, as a consequence of the inequalities mentioned above. The limit superior is finite also for the Douady-Earle extension, although an explicit value seems not be available in the literature.

It is also worth mentioning that Corollary \ref{cor ph to 0} highlights that the lower bound \eqref{eq inequality intro2}, obtained in \cite{seppimaximal} by means of geometric methods using Anti-de Sitter geometry, is at least not far from being optimal, since $2/\pi\approx 0.64$.

On the other hand, when $||\ph||_{cr}\to +\infty$, we obtain the following corollary:

\begin{corx} \label{cor ph to infty}
If $f_\ph:\Hyp^2\to\Hyp^2$ denotes the minimal Lagrangian extension of a quasisymmetric homeomorphism $\ph: \RP^1\to\RP^1$, then:
$$\liminf_{||\ph||_{cr}\to +\infty}\frac{\log K(f_\ph)}{||\ph||_{cr}}=0\qquad\text{and}\qquad \limsup_{||\ph||_{cr}\to +\infty}\frac{\log K(f_\ph)}{||\ph||_{cr}}\in\left[\frac{\sqrt 2}{2},+\infty\right)~.$$
\end{corx}

Again, the reader might want to compare the aforementioned inequality for the Beurling-Ahlfors extension, which shows that the limit superior of $\log K(f_\ph^{B\!A})/||\ph||_{cr}$ is at most $1$ as $||\ph||_{cr}\to +\infty$. Finally, let us observe that from the work of Strebel (\cite{strebel1,strebel2}), it is known that the extremal extension $f_{\ph_\lambda}^e$ (namely, the extension which minimizes the maximal dilatation) of the simple earthquake $\ph_\lambda$ satisfies $K(f_{\ph_\lambda}^e)=O(\lambda^2)$ as $\lambda\to +\infty$. Hence a consequence of Theorem \ref{thm simple earth k=0} is the following:

\begin{corx} \label{cor exp far}
The minimal Lagrangian extension of the simple earthquake $\ph_\lambda$ stays exponentially far from being extremal as $\lambda\to +\infty$.
\end{corx}

Compare also \cite[Theorem 5]{hu_muzician_douadyearle} for the Douady-Earle extension.

\subsection*{Discussion of the techniques}

The proof of the inequality \eqref{eq inequality intro}, as well as the proof of the existence and uniqueness result of Bonsante and Schlenker, involves the use of \emph{Anti-de Sitter} three-dimensional geometry, and in particular the study of \emph{maximal surfaces}. See also \cite{ksurfaces} which provides an alternative proof of the existence, \cite{bbzads,Schlenker-Krasnov,tamburellicmc,bonsepequivariant,toulisse,seppiminimal,barbotkleinian} for related results on (maximal) surfaces, \cite{Mess} as a standard reference on three-dimensional Anti-de Sitter geometry, and \cite{questionsads} for a survey on related questions. On the other hand, we remark that the results of this paper do not rely on Anti-de Sitter geometry, and their setting might also be compared with \cite{toulisse2}.

We will not enter into technical details here, but we mention the idea behind the construction of the two families of minimal Lagrangian extensions which appear in Theorem \ref{thm simple earth k=0}  and Theorem \ref{thm power a=0}. Instead of extending a given quasisymmetric homeomorphism, we rather start from (two) \emph{ansatz} which lead to (two) remarkable families of minimal Lagrangian diffeomorphisms. These ansatz basically consist in looking for diffeomorphisms $f:\Hyp^2\to\Hyp^2$ with a certain invariance with respect to 1-parameter families of isometries of $\Hyp^2$. Such invariance enables to reduce the condition that $f$ is minimal Lagrangian to ordinary differential equations. Solutions of such differential equations are then studied analytically, and the corresponding diffeomorphisms are interpreted from the point of view of hyperbolic geometry.
 This enables to estimate the maximal dilatation of $f$ on the one hand, and on the other the cross-ratio distortion of the boundary values of $f$ (which are recognized as the simple left earthquakes in one case, and the power functions in the other).
 
Finally, we remark that the two one-parameter families $f_{\ph_\lambda}$ and $f_{\psi_\alpha}$ are subfamilies of a promising two-parameter family, whose boundary value is the quasisymmetric homeomorphism
$$x\mapsto=\begin{cases} -|x|^\alpha & \textrm{if } x\leq 0 \\ e^{\lambda} x^\alpha & \textrm{if }x> 0 \end{cases}~,$$
for parameters $\lambda\in[0,+\infty)$ and $\alpha\in(0,+\infty)$. Clearly, this reduces to $\ph_\lambda$ when $\alpha=1$ and to $\psi_\alpha$ when $\lambda=0$. The understanding of the minimal Lagrangian extensions for such more general two-parameter family might lead to more precise estimates in the spirit of Corollary \ref{cor ph to 0} and Corollary \ref{cor ph to infty}. However, a description of this  two-parameter family of minimal Lagrangian diffeomorphisms seems to be difficult to achieve with the techniques used in this paper, and is therefore not tackled here. 

\subsection*{Organization of the paper}

Section \ref{sec preliminaries} contains preliminary results on quasiconformal mappings, minimal Lagrangian diffeomorphisms of $\Hyp^2$, quasisymmetric homeomorphisms, and state the necessary results known from the literature. Section \ref{sec earth simple} studies the first one-parameter family of minimal Lagrangian diffeomorphisms, leading to Theorem \ref{thm simple earth k=0}, whereas Section \ref{sec power} considers the second family and leads to Theorem \ref{thm power a=0}. Finally, Section \ref{sec conclusions} contains the Corollary \ref{cor ph to 0}, \ref{cor ph to infty} and \ref{cor exp far} and compares them with results, known from the literature, for the more classical extensions.

\subsection*{Acknowledgements}

I am extremely grateful to an anonymous referee for several comments which highly improved the present article. I would like to thank Jean-Marc Schlenker for motivating me towards questions of this type since several years, and Francesco Bonsante and Jun Hu for many interesting discussions on related topics.

\section{Preliminaries} \label{sec preliminaries}

In this paper, we consider the hyperbolic plane in the upper half-plane model, namely:
$$\Hyp^2=\left(\left\{z=x+iy\in\C\,:\,y>0\right\},\frac{dx^2+dy^2}{y^2}\right)~,$$
and we identify its visual boundary $\partial_\infty\Hyp^2$ with $\R\mathrm P^1\cong \R\cup\{\infty\}$, which in this model is the projective line $\{y=0\}\cup\{\infty\}$. Recall that, in this model, the group of orientation-preserving isometries of $\Hyp^2$ is isomorphic to $\PSL_2(\R)$, acting on $\C\mathrm P^1\cong\C\cup\{\infty\}$ by M\"obius transformations which preserve the real line $\R\mathrm P^1$.

\subsection*{Quasisymmetric homeomorphisms}

We will adopt the following definition of cross-ratio:
\begin{equation} \label{eq cross ratio}
cr(x_1,x_2,x_3,x_4)=\frac{(x_4-x_1)(x_3-x_2)}{(x_2-x_1)(x_3-x_4)}\,,
\end{equation}
for a quadruple $x_1,x_2,x_3,x_4\in\RP^1\cong\R\cup\{\infty\}$. We say that a quadruple  $Q=(x_1,x_2,x_3,x_4)$ of points of $\R\cup\{\infty\}$ is \emph{symmetric} if the geodesics $\ell,\ell'$ of $\Hyp^2$ connecting $x_1$ to $x_3$ and $x_2$ to $x_4$ intersect orthogonally. Composing with an element of $\PSL_2(\R)$, such a quadruple is equivalent to $(-1,0,1,\infty)$ and therefore has cross-ratio equal to $-1$.

\begin{defi} \label{defi cross ratio norm}
Let $\ph:\RP^1\to\RP^1$ be an orientation-preserving homeomorphism. Then $\ph$ is \emph{quasisymmetric} if 
$$||\ph||_{cr}:=\sup_{cr(Q)=-1}\left|\log\left|cr(\ph(Q))\right|\right|<+\infty\,.$$
In this case, the quantity $||\ph||_{cr}$ is called \emph{cross-ratio norm} of $\ph$.
\end{defi}

It can be shown that $||\ph||_{cr}=0$ if and only if $\ph\in\PSL_2(\R)$. In this case, $\ph$ admits a conformal extension to the upper half-plane $\Hyp^2$, which is defined by the action of the same element of $\PSL_2(\R)$ on $\Hyp^2$ by a M\"obius transformation.

 \subsection*{Quasiconformal diffeomorphisms}
 Ahlfors and Beurling showed that every quasisymmetric homeomorphism $\ph$ admits a \emph{quasiconformal} extension $f:\Hyp^2\to\Hyp^2$ and conversely, every quasiconformal map $f:\Hyp^2\to\Hyp^2$ extends continuously to a quasisymmetric homeomorphism $\ph:\RP^1\to\RP^1$.

It turns out, for instance from Theorem \ref{thm bon schl} below, that a quasisymmetric homeomorphism always admits a \emph{smooth} quasiconformal extension $f:\Hyp^2\to\Hyp^2$. (That is, $f$ is a diffeomorphism of $\Hyp^2$ which extends continuously to $\RP^1$.) We thus provide the definition of quasiconformality only for diffeomorphisms, although there is a more general definition with weaker regularity. (See for instance \cite{ahlforsbook, gardiner2,fletchermarkoviclibro}.)

\begin{defi} \label{defi quasiconformal}
Let $f:\Hyp^2\to\Hyp^2$ be an orientation-preserving diffeomorphism. Then $f$ is \emph{quasiconformal} if 
$$K(f):=\sup_{z\in\C}\frac{\text{major axis of }df_z(\mathbb S^1)}{\text{minor axis of }df_z(\mathbb S^1)}<+\infty~,$$
where we denote $\mathbb S^1=\{|z|=1\}\subset T_z\Hyp^2=\C$, and $df_z$ is the differential of $f$ at the point $z$. In this case, the quantity $K(f)$ is called \emph{maximal dilatation} of $f$.
\end{defi}

Roughly speaking, a quasiconformal diffeomorphism maps, at first order, circles to ellipses of bounded eccentricity. Clearly, $K(f)\geq 1$. It turns out that $K(f)=1$ if and only if $f\in\PSL_2(\R)$. 

\subsection*{Minimal Lagrangian extensions}
In this paper, we will study the properties of the \emph{minimal Lagrangian diffeomorphism} extending $\ph$, which is defined as follows:

\begin{defi} \label{defi min lag}
Let $f:\Hyp^2\to\Hyp^2$ be an orientation-preserving diffeomorphism. Then $f$ is \emph{minimal Lagrangian} if it is area-preserving and its graph is a minimal surface in the Riemannian product $\Hyp^2\times\Hyp^2$.
\end{defi}

An important characterization of minimal Lagrangian diffeomorphisms is the following. Let us denote by $g_{\Hyp^2}$ the hyperbolic metric of $\Hyp^2$.

\begin{prop}[{\cite{labourieCP},\cite[Proposition 1.2.6]{jeremythesis}}] \label{prop char min lag}
An orientation-preserving diffeomorphism $f:\Hyp^2\to\Hyp^2$ is minimal Lagrangian if and only if $f^*g_{\Hyp^2}=g_{\Hyp^2}(b\cdot,b\cdot)$, where $b\in\Gamma(\mathrm{End}(T\Hyp^2))$ is a bundle morphism such that:
\begin{itemize}
\item $b$ is self-adjoint for $g_{\Hyp^2}\,;$
\item $\det b=1\,;$
\item $d^{\nabla}b=0\,$.
\end{itemize}
Here $\nabla$ is the Levi-Civita connection of $g_{\Hyp^2}$, and $d^\nabla b$ is the exterior derivative defined by:
$$(d^\nabla b)(v,w)=\nabla_v (b(w))-\nabla_w (b(v))-b[v,w]~,$$
where $v$ and $w$ are any local smooth vector fields.
\end{prop}

\begin{remark} \label{rmk bundle morphism}
Given an orientation-preserving diffeomorphism $f:\Hyp^2\to\Hyp^2$, there is a unique section $b\in\Gamma(\mathrm{End}(T\Hyp^2))$, satisfying $f^*g_{\Hyp^2}=g_{\Hyp^2}(b\cdot,b\cdot)$, which is self-adjoint for $g_{\Hyp^2}$ and positive definite. In fact, if $b$ is self-adjoint, then $f^*g_{\Hyp^2}=g_{\Hyp^2}(b\cdot,b\cdot)=g_{\Hyp^2}(\cdot,b^2\cdot)$ and therefore (in matrix notation), 
\begin{equation} \label{eq square b}
b^2=g_{\Hyp^2}^{-1}\cdot f^*g_{\Hyp^2}~.
\end{equation}
 Such tensor $b^2$ then has two self-adjoint square roots, one positive definite and one negative definite, say $b$ and $-b$. The condition of being minimal Lagrangian thus consist in imposing that one of these sections (and thus also the other) is Codazzi and has unit determinant.
\end{remark}

We can finally state the main results which motivate this paper. First, the existence and uniqueness of the minimal Lagrangian extension, which is due to Bonsante and Schlenker. See also \cite{ksurfaces} for an alternative proof.

\begin{theorem}[\cite{bon_schl}] \label{thm bon schl}
Let $\ph:\RP^1\to\RP^1$ be a quasisymmetric homeomorphism. Then there exists a unique minimal Lagrangian extension $f_\ph:\Hyp^2\to\Hyp^2$, which is quasiconformal.
\end{theorem}

We observe that, if $\ph$ is in $\PSL_2(\R)$, then the minimal Lagrangian extension $f_\ph$ is the natural extension, defined by the action of the same element of $\PSL_2(\R)$ on $\Hyp^2$. (In fact, in this case $f_\ph$ is an isometry for $g_{\Hyp^2}$, and therefore one can take $b$ as the identity section in Definition \ref{defi min lag}.)
The maximal dilatation of the minimal Lagrangian extension was studied in \cite{seppimaximal}.

\begin{theorem}[\cite{seppimaximal}] \label{thm seppi}
There exists a universal constant $C>0$ such that
$$\log K(f_\ph)\leq C||\ph||_{cr}~,$$
for any quasisymmetric homeomorphism $\ph:\RP^1\to\RP^1$.
\end{theorem}
In other words, Theorem \ref{thm seppi} can be stated as:
$$\sup_{||\ph||_{cr}\neq 0}\frac{\log K(f_\ph)}{||\ph||_{cr}}<+\infty~.$$
The constant $C$ obtained in \cite{seppimaximal} is not explicit, though.
Moreover, in the same paper, the following (partial) converse inequality was obtained:
\begin{theorem}[\cite{seppimaximal}] \label{thm seppi 2}
For every $\epsilon>0$ there exists $\delta>0$ such that, if $\ph:\RP^1\to\RP^1$ is quasisymmetric with $||\ph||_{cr}\leq \delta$, then
$$\log K(f_\ph)\geq \left(\frac{1}{2}-\epsilon\right)||\ph||_{cr}~.$$
\end{theorem}
\noindent Theorem \ref{thm seppi 2} can be reformulated as
\begin{equation} \label{eq seppi 2}
\liminf_{||\ph||_{cr}\to 0}\frac{\log K(f_\ph)}{||\ph||_{cr}}\geq \frac{1}{2}~.
\end{equation}

\section{Extending simple earthquake maps} \label{sec earth simple}

In this section we will construct a first one-parameter family of minimal Lagrangian maps $\f_c:\Hyp^2\to\Hyp^2$, depending on a parameter  $c\in[0,1)$, and study its behavior.

\subsection*{The band model} It will be convenient to introduce a new coordinate system on $\Hyp^2$, which is the \emph{band model} of $\Hyp^2$. We use $z=x+iy$ to denote points of $\Hyp^2$, and introduce the notation $w=u+iv$. Then consider the global coordinate $\varsigma:\Omega\to\Hyp^2$, where
$$\Omega=\left\{w=u+iv\,:\,|u|<\frac{\pi}{2}\right\}~,$$
defined by
$$\varsigma:w\mapsto z=i\exp(-iw)~.$$
Observe that $\varsigma$ extends to each component $\{u=-\pi/2\}$ and $\{u=\pi/2\}$ of $\partial\Omega$, which are mapped to the two portions $\{x>0\}$ and $\{x<0\}$ of $\RP^1$ by means of 
$$\varsigma\left(-\frac{\pi}{2}+iv\right)=-e^v\qquad \varsigma\left(\frac{\pi}{2}+iv\right)=e^v~.$$
In the $w$-coordinates, the metric of $\Hyp^2$ takes the form:
$$\varsigma^*g_{\Hyp^2}=\frac{du^2+dv^2}{\cos^2(u)}~.$$
Let us thus remark that this is a conformal model of $\Hyp^2$, for which the vertical line $\ell=\{u=0\}$ is a geodesic with arclength parameter $v$. The endpoints of $\ell$ are at infinity and correspond to $0$ and $\infty$ in the upper half-plane model. In fact, the hyperbolic translations along $\ell$, with (signed) translation distance $\delta$, have the form 
$$T_\delta(w)= w+i\delta~.$$
See also \cite[Chapter 2]{hubbardbook}.

\subsection*{Ansatz for minimal Lagrangian diffeomorphisms} For convenience, let us use the notation 
$$w=(u,v)\in \left(-\frac{\pi}{2},\frac{\pi}{2}\right)\times \R$$
for the band model of $\Hyp^2$, instead of the complex notation $w=u+iv$. In this section we will consider diffeomorphisms $\f:\Hyp^2\to\Hyp^2$ of the following form:
\begin{equation} \label{eq ansatz 1}
\f(u,v)=(u,v+h(u))~.
\end{equation}
We will derive conditions on $h$ which correspond to the fact that $\f$ is minimal Lagrangian, using the characterization of Proposition \ref{prop char min lag}, and study the solutions of the corresponding ODE. As observed in Remark \ref{rmk bundle morphism}, the section $b\in\Gamma(\mathrm{End}(T\Hyp^2))$ is essentially determined by the map $\f$, and $\f$ being minimal Lagrangian corresponds to imposing $\det b=1$ and $d^\nabla b=0$. We will see below that the condition $\det b=1$ will be easily satisfied, while the most complicated condition is the Codazzi condition $d^\nabla b=0$.

\begin{remark}
The diffeomorphism of the form of Equation \eqref{eq ansatz 1} is invariant for the 1-parameter group of isometries $\{T_\delta\,:\,\delta\in\R\}$, that is, 
$$\f\circ T_\delta=T_\delta\circ \f~,$$
for every $\delta\in \R$.
A diffeomorphism with this invariance can have the more general form
\begin{equation} \label{eq ansatz 1 alternative}
\f(u,v)=(g(u),v+h(u))~.
\end{equation}
Such maps were considered, in a slightly different context, in \cite[Section 6]{ksurfaces}. There it was showed that, when taking $\f$ of the form \eqref{eq ansatz 1 alternative}, if $b$ is the unique positive definite self-adjoint bundle morphism with $f^*g_{\Hyp^2}=g_{\Hyp^2}(b\cdot,b\cdot)$ (see Remark \ref{rmk bundle morphism}), the condition $\det b=1$ is equivalent to the following ODE on $g$:
$$\frac{g'(u)}{\cos^2(g(u))}=\pm\frac{1}{\cos^2(u)}~,$$
 which is solved by $\tan(g(u))=\pm\tan u+C$. Since in this work we are looking for global diffeomorphisms, we impose $C=0$, so that $g$ is defined on the entire interval $(-\pi/2,\pi/2)$, and we take the positive sign since we require $\f$ to be orientation-preserving. For this reason, we only consider $g(u)=u$ here, as in Equation \eqref{eq ansatz 1}.
\end{remark}

Let us now compute the self-adjoint, positive definite section $b\in\Gamma(\mathrm{End}(T\Hyp^2))$, as explained in Remark \ref{rmk bundle morphism}. 

\begin{lemma} \label{lemma b square root}
Let $\f:\Hyp^2\to\Hyp^2$ be an orientation-preserving diffeomorphism  of the form (in the band model)
$$\f(u,v)=(u,v+h(u))~.$$
Then $\f^*g_{\Hyp^2}=g_{\Hyp^2}(b\cdot,b\cdot)$ where
$$b=\frac{1}{\sqrt{4+h'(u)^2}}\begin{pmatrix} 2+h'(u)^2 & h'(u) \\ h'(u) & 2 \end{pmatrix}$$
is self-adjoint, positive definite, and $\det b=1$.
\end{lemma}
\begin{proof}
Let us first compute the pull-back metric, in the band model:
\begin{align*}
\f^*g_{\Hyp^2}=&\frac{1}{\cos^2(u)}\left(du^2+(dv+h'(u)du)^2\right) \\
=&\frac{(1+h'(u)^2)du^2+2h'(u)dudv+dv^2}{\cos^2(u)}\,.
\end{align*}
Using Equation \eqref{eq square b}, and recalling that in these coordinates $g_{\Hyp^2}=(1/\cos^2(u))(du^2+dv^2)$, we then have:
$$b^2=\begin{pmatrix} 1+h'(u)^2 & h'(u) \\ h'(u) & 1 \end{pmatrix}~.$$
Hence $\det b^2=1$. To compute $b$, self-adjoint for $g_{\Hyp^2}$ and positive definite, we observe from the Cayley-Hamilton theorem that $b^2-(\tr b)b+(\det b)\mathbbm 1=0$. Since we are looking for $b$ with $\det b=1$, we deduce $(\tr b)b=b^2+\mathbbm 1$. By taking the trace we obtain $(\tr b)^2=\tr b^2+2$ and therefore $\tr b=\sqrt{2+\tr b^2}=\sqrt{4+h'(u)^2}$. (Indeed $b$ is supposed positive definite and thus $\tr b>0$.) This enables us to conclude
$$b=\frac{1}{\sqrt{4+h'(u)^2}}(b^2+\mathbbm 1)=\frac{1}{\sqrt{4+h'(u)^2}}\begin{pmatrix} 2+h'(u)^2 & h'(u) \\ h'(u) & 2 \end{pmatrix}~,$$
as claimed.
\end{proof}

\subsection*{Imposing the Codazzi condition}
We are thus left to impose the condition $d^\nabla b=0$, which will result in an ordinary differential equation for $h$. 

\begin{lemma} \label{lemma impose codazzi}
Let $\f:\Hyp^2\to\Hyp^2$ be an orientation-preserving diffeomorphism  of the form (in the band model)
$$\f(u,v)=(u,v+h(u))~,$$
and let $\f^*g_{\Hyp^2}=g_{\Hyp^2}(b\cdot,b\cdot)$ where $b$ is 
self-adjoint, positive definite and $\det b=1$, as in Lemma \ref{lemma b square root}. Then $d^\nabla b=0$ if and only if $h$ satisfies
\begin{equation} \label{eq impose codazzi}
\tan(u)=-\frac{2h''(u)}{(4+h'(u)^2)h'(u)}~.
\end{equation}
\end{lemma}
\begin{proof}
First, let us remark that $d^\nabla b$ is a skew-symmetric 2-form, hence in $(u,v)$-coordinates, $d^\nabla b(\partial_u,\partial_u)=d^\nabla b(\partial_v,\partial_v)=0$ and $d^\nabla b(\partial_u,\partial_v)=-d^\nabla b(\partial_v,\partial_u)$. Moreover $[\partial_u,\partial_v]=0$. It thus suffices to check 
$$d^\nabla b(\partial_u,\partial_v)=\nabla_{\partial_u}b(\partial_v)-\nabla_{\partial_v}b(\partial_u)=0~.$$
Let us now compute:
\begin{align*}
\nabla_{\partial_v}b(\partial_u)&=\nabla_{\partial_v}\left(\frac{1}{\sqrt{4+h'(u)^2}}((2+h'(u)^2)\partial_u)+h'(u)\partial_v\right) \\
&=\frac{1}{\sqrt{4+h'(u)^2}}\left((2+h'(u)^2)\nabla_{\partial_v}\partial_u+h'(u)\nabla_{\partial_v}\partial_v\right) \\
&=\frac{1}{\sqrt{4+h'(u)^2}}\left((2+h'(u)^2)\tan(u)\partial_v-h'(u)\tan(u)\partial_u\right) \\
&=\frac{\tan(u)}{\sqrt{4+h'(u)^2}}\left(-h'(u)\partial_u+(2+h'(u)^2)\partial_v\right)~.
\end{align*}
From the second to the third line we have used the expressions $$\nabla_{\partial_v}\partial_u=\Gamma_{vu}^u\partial_u+\Gamma_{vu}^v\partial_v=\tan(u)\partial_v$$ and 
$$\nabla_{\partial_v}\partial_v=\Gamma_{vv}^u\partial_u+\Gamma_{vv}^v\partial_v=-\tan(u)\partial_u~,$$ which arise from a direct computation of the Christoffel symbols: it turns out indeed that $\Gamma_{vu}^u=\Gamma_{vv}^v=0$, $\Gamma_{vu}^v=\tan(u)$ and $\Gamma_{vv}^u=-\tan(u)$.

On the other hand, let us compute:
\begin{align*}
\nabla_{\partial_u}b(\partial_v)=&\nabla_{\partial_u}\left(\frac{1}{\sqrt{4+h'(u)^2}}(h'(u)\partial_u+2\partial_v)\right) \\
=&-\frac{h'(u)h''(u)}{(4+h'(u)^2)^{3/2}}(h'(u)\partial_u+2\partial_v) \\
&+\frac{1}{\sqrt{4+h'(u)^2}}(h''(u)\partial_u+h'(u)\nabla_{\partial_u}\partial_u+2\nabla_{\partial_u}\partial_v) \\
=& \left(-\frac{h'(u)^2h''(u)}{(4+h'(u)^2)^{3/2}}+\frac{h'(u)\tan(u)+h''(u)}{\sqrt{4+h'(u)^2}}\right)\partial_u \\
&+\left(-\frac{2h'(u)h''(u)}{(4+h'(u)^2)^{3/2}}+\frac{2\tan(u)}{\sqrt{4+h'(u)^2}}\right)\partial_v~,
\end{align*}
where here we have used 
$$\nabla_{\partial_u}\partial_u=\Gamma_{uu}^u\partial_u+\Gamma_{uu}^v\partial_v=\tan(u)\partial_u$$
since $\Gamma_{uu}^u=\tan(u)$ and $\Gamma_{uu}^v=0$, 
 and $$\nabla_{\partial_u}\partial_v=\nabla_{\partial_v}\partial_u=\tan(u)\partial_v~.$$ Let us now equate $\nabla_{\partial_v}b(\partial_u)$ and $\nabla_{\partial_u}b(\partial_v)$. 
From the $\partial_u$ component, we obtain:
$$
\frac{h'(u)^2h''(u)}{(4+h'(u)^2)^{3/2}}=\frac{2h'(u)\tan(u)+h''(u)}{\sqrt{4+h'(u)^2}}~,
$$
which leads after some manipulation to:
$$
\tan(u)=-\frac{2h''(u)}{(4+h'(u)^2)h'(u)}~,
$$
which is precisely Equation \eqref{eq impose codazzi}.
Equating the $\partial_v$ components instead, one gets to:
$$
\frac{2h'(u)h''(u)}{(4+h'(u)^2)^{3/2}}=-\frac{h'(u)\tan(u)}{\sqrt{4+h'(u)^2}}~,
$$
which leads again to Equation \eqref{eq impose codazzi}. Hence $d^\nabla b=0$ if and only if Equation \eqref{eq impose codazzi}
 is satisfied.
 \end{proof}

\subsection*{Solutions of ODE and definition of the diffeomorphisms $\f_k$}

We thus need to study solutions of Equation \eqref{eq impose codazzi}. This is not difficult to integrate, since 
$$\int\tan(u)du=-\log\cos(u)+c_1$$ and $$\int\frac{2h''(u)}{(4+h'(u)^2)h'(u)}du=\frac{1}{2}\log h'(u)-\frac{1}{4}\log(4+h'(u)^2)+c_2~,$$
and therefore we obtain (put $c=-c_1-c_2$):
$$e^{4c}\cos^4(u)=\frac{h'(u)^2}{4+h'(u)^2}~.$$
After some manipulation, and replacing the constant $e^{c}$ by $k$, we obtain
\begin{equation} \label{eq hk'}
h'(u)=\pm\frac{2k^2\cos^2(u)}{\sqrt{1-k^4\cos^4(u)}}~.
\end{equation}

Composing with a hyperbolic translation $T_\delta$, which has the form $T_\delta(u,v)=(u,v+\delta)$, we can assume $h(-\pi/2)=0$. Hence we will consider the following family of minimal Lagrangian diffeomorphisms (corresponding to the choice of the positive sign in Equation \eqref{eq hk'}):
\begin{defi}
For every $k\in[0,1)$, let 
$$h_k(u)=\int_{-\pi/2}^u \frac{2k^2\cos^2(\sigma)}{\sqrt{1-k^4\cos^4(\sigma)}}d\sigma~.$$
Then define, in the band model of $\Hyp^2$,
 $$\f_k(u,v)=(u,v+h_k(u))~.$$
\end{defi}

The diffeomorphisms $\f_k$ obtained in this way
(recall the ansatz of Equation \eqref{eq ansatz 1})
are minimal Lagrangian, since $h_k$ is a solution of Equation \eqref{eq impose codazzi} in Lemma \ref{lemma impose codazzi}.
They depend on the parameter $k$. We only consider $k<1$ so that $h_k$ is defined for all $u\in(-\pi/2,\pi/2)$. Moreover, we also allowed $k=0$, which corresponds to $h_0'\equiv 0$ and thus $h_0\equiv 0$: the diffeomorphism $\f_0$ is the identity.

\subsection*{Computing the cross-ratio norm}
Let us first express the boundary value of the diffeomorphisms $\f_k$ we have just defined, and find its cross-ratio norm.

\begin{lemma}
The minimal Lagrangian diffeomorphisms $\f_k:\Hyp^2\to \Hyp^2$ extend to $\RP^1$, with boundary value (in the upper half-plane model) the homeomorphism: 
$$\phi_k(x)=\begin{cases} x & \textrm{if }x\leq 0 \\ e^{\lambda_k} x & \textrm{if }x> 0 \end{cases}~,$$
where 
\begin{equation} \label{eq lambda integral}
\lambda_k=\int_{-\pi/2}^{\pi/2} \frac{2k^2\cos^2(\sigma)}{\sqrt{1-k^4\cos^4(\sigma)}}d\sigma~.
\end{equation}
\end{lemma}
\begin{proof}
Since we assumed (in the band model coordinates) $h_k(-\pi/2)=0$, the diffeomorphism $\f_k$ extends to the boundary, with 
$$\f_k(-\pi/2,v)=(-\pi/2,v)\qquad \f_k(\pi/2,v)=(\pi/2,v+h_k(\pi/2))~.$$
In other words, in the upper half-plane model, $\f_k$ is the identity on the portion $\{x\leq 0\}$ of $\R\cup\{\infty\}$, whereas it coincides on the portion $\{x>0\}$ with the restriction on $\R\cup\{\infty\}$ of the hyperbolic transformation whose axis is the geodesic of the upper half-plane with endpoints $0$ and $\infty$, and having (signed) translation distance 
\begin{equation} \label{eq lambda integral}
\lambda_k=h_k(\pi/2)=\int_{-\pi/2}^{\pi/2} \frac{2k^2\cos^2(\sigma)}{\sqrt{1-k^4\cos^4(\sigma)}}d\sigma~.
\end{equation}
Hence, in the upper half-plane model, the boundary value of $\f_k$ is:
$$\phi_k(x)=\begin{cases} x & \textrm{if }x\leq 0 \\ e^{\lambda_k} x & \textrm{if }x> 0 \end{cases}~.$$
This concludes the proof.
\end{proof}

\begin{remark} \label{rmk cross ratio simple earth}
In the language of hyperbolic geometry, $\phi_k$ is the boundary value of a \emph{left earthquake map}, whose earthquake lamination is the geodesic line $\{x=0\}$ with weight $\lambda_k$.
To simplify the notation, let us denote the restriction to $\partial\Hyp^2$ of a simple left earthquake map with weight $\lambda$ by $\ph_\lambda$ -- namely, we put $\lambda=\lambda_k$ and we denote $\ph_\lambda=\phi_k$. Using Definition \ref{defi cross ratio norm}, it can be easily verified that 
$$||\ph_\lambda||_{cr}=\lambda~.$$
In fact, picking the quadruple $Q=(x_1,x_2,x_3,x_4)$ with:
$$x_1=-a\qquad x_2=0\qquad x_3=a\qquad x_4=\infty$$
for some $a>0$, by a direct computation $cr(Q)=-1$. On the other hand,
$$\phi_k(x_1)=-a\qquad \phi_k(x_2)=0\qquad \phi_k(x_3)=e^{\lambda} a\qquad \phi_k(x_4)=\infty~,$$
hence by another direct computation $cr(\phi_k(Q))=-e^{\lambda}$.
Therefore
$$||\ph_\lambda||_{cr}:=\sup_{cr(Q)=-1}\left|\log\left|cr(\ph(Q))\right|\right|\geq \lambda\,.$$
It is not difficult to check that $\left|\log\left|cr(\ph(Q))\right|\right|$ is actually maximized for quadruples of this form, and therefore $||\ph_\lambda||_{cr}=\lambda$.
See also \cite{hu_muzician_douadyearle}. Let us also remark that 
$$\lim_{k\to 0^+}||\phi_k||_{cr}=\lim_{\lambda\to 0^+}||\ph_\lambda||_{cr}=0~.$$
We will study more precisely the behavior of $||\phi_k||_{cr}$ as $k\to 0$ below, where we will also see (with a precise asymptotic behavior) that
$$\lim_{k\to 1^-}||\phi_k||_{cr}=+\infty~.$$
\end{remark}

\subsection*{Computing the maximal dilatation} Let us moreover compute the maximal dilatation of the minimal Lagrangian diffeomorphism $\f_k:\Hyp^2\to\Hyp^2$. 

\begin{lemma} \label{lemma max dil simple}
The maximal dilatation of the minimal Lagrangian diffeomorphisms $\f_k:\Hyp^2\to\Hyp^2$ is:
$$K(\f_k)={\frac{1+k^2}{1-k^2}}~.$$
\end{lemma}
\begin{proof}
Recalling Definition \ref{defi quasiconformal} and the expression 
$$f^*g_{\Hyp^2}=g_{\Hyp^2}(b\cdot,b\cdot)~,$$
for any point $z\in\Hyp^2$ the ratio between the major axis and the minor axis of the ellipse $df_z(\mathbb S^1)$ equals ${\eta_+(z)}/{\eta_-(z)}$, where $\eta_+(z)$ and $\eta_-(z)$ are the largest and smallest eigenvalues of $b$ at $z$. Since $\det b=1$, $\eta_+\eta_-=1$ and therefore the maximal dilatation of $f$ is:
$$K(f)=\sup_{z\in\Hyp^2} \eta_+(z)^2~.$$ 
In the case of the minimal Lagrangian diffeomorphisms $\f_k$, since $\det b=1$ and $\tr b=\sqrt{4+h_k'(u)}$, the eigenvalues of $b$ are:
$$\eta_\pm(u,v)=\frac{\sqrt{4+h_k'(u)^2}\pm h_k'(u) }{2}~.$$
From Equation \eqref{eq hk'}, 
$$h_k'(u)=\frac{2k^2\cos^2(u)}{\sqrt{1-k^4\cos^4(u)}}$$
achieves its maximum at $u=0$, namely for:
$$h_k'(0)=\frac{2k^2}{\sqrt{1-k^4}}~.$$
Hence we finally have
\begin{align*}
K(\f_k)&=\left(\frac{\sqrt{4+h_k'(0)^2}+ h_k'(0) }{2}\right)^2 \\
&=\frac{1}{4}\left(\sqrt{4+\frac{4k^4}{1-k^4}}+\frac{2k^2}{\sqrt{1-k^4}}\right)^2 \\
&=\frac{1}{4}\left(\frac{2(1+k^2)}{\sqrt{1-k^4}}\right)^2=\left(\frac{1+k^2}{\sqrt{1-k^4}}\right)^2={\frac{1+k^2}{1-k^2}}~,
\end{align*}
as claimed.
\end{proof}

\subsection*{Comparison when $k\to 0$}

Let us now study the value of the ratio $\log K(\f_k)/||\phi_k||_{cr}$, in particular when $k\to 0$ and $k\to 1$. 

\begin{reptheorem}{thm simple earth k=0}[Case $\lambda\to 0$]
Let $f_{\ph_\lambda}:\Hyp^2\to\Hyp^2$ be the minimal Lagrangian diffeomorphisms which extend a simple left earthquake $\ph_\lambda$ of weight $\lambda$. Then
$$\lim_{\lambda\to 0}\frac{\log K(f_{\ph_\lambda})}{||\ph_\lambda||_{cr}}=\frac{2}{\pi}\approx 0.64~.$$
\end{reptheorem}
\begin{proof}
First, recall that the cross-ratio norm of ${\ph_\lambda}$ is $||\ph_\lambda||_{cr}=\lambda$, from Remark \ref{rmk cross ratio simple earth}. Since $\phi_k=\ph_{\lambda_k}$, it follows that $f_{\ph_\lambda}=\f_k$ is the minimal Lagrangian homeomorphism extending a simple left earthquake of weight $\lambda_k$. As $||\phi_k||_{cr}=\lambda_k\to 0$ when $k\to 0$, we have
$$\lim_{\lambda\to 0^+}\frac{\log K(f_{\ph_\lambda})}{||\ph_\lambda||_{cr}}=\lim_{k\to 0^+}\frac{\log K(\f_k)}{||\phi_k||_{cr}}~.$$

We first need to estimate $||\phi_k||_{cr}=\lambda_k$, which is defined by the integral of Equation \eqref{eq lambda integral}. Since ${1-k^4\cos^4(u)}=(1-k^2\cos^2(u))(1+k^2\cos^2(u))$ and 
$$1\leq \sqrt{1+k^2\cos^2(u)}\leq \sqrt{1+k^2}~,$$
we can estimate
\begin{equation} \label{eq double ineq integral}
\frac{1}{\sqrt{1+k^2}}\int_{-\pi/2}^{\pi/2} \frac{2k^2\cos^2(\sigma)}{\sqrt{1-k^2\cos^2(\sigma)}}d\sigma\leq \lambda_k\leq \int_{-\pi/2}^{\pi/2} \frac{2k^2\cos^2(\sigma)}{\sqrt{1-k^2\cos^2(\sigma)}}d\sigma~.
\end{equation}
From some algebraic steps (in the second line we change variable $\theta=\pi/2-\sigma$),
\begin{align*}
\int_{-\pi/2}^{\pi/2} \frac{2k^2\cos^2(\sigma)}{\sqrt{1-k^2\cos^2(\sigma)}}d\sigma&=4\int_{0}^{\pi/2} \frac{k^2\cos^2(\sigma)}{\sqrt{1-k^2\cos^2(\sigma)}}d\sigma \\
&=4\int_{0}^{\pi/2} \frac{k^2\sin^2(\theta)}{\sqrt{1-k^2\sin^2(\theta)}}d\theta \\
&=4\left(\int_{0}^{\pi/2} \frac{d\theta}{\sqrt{1-k^2\sin^2(\theta)}}-\int_{0}^{\pi/2} {\sqrt{1-k^2\sin^2(\theta)}}d\theta\right) \\
&=4(K(k)-E(k))~,
\end{align*}
where 
$$K(k)=\int_{0}^{\pi/2} \frac{d\theta}{\sqrt{1-k^2\sin^2(\theta)}}$$
 denotes the complete elliptic integral of the first kind (there is a small abuse of notation, since $K$ also denotes the maximal dilatation, but this is the standard notation in both cases), and 
 $$E(k)=\int_{0}^{\pi/2} {\sqrt{1-k^2\sin^2(\theta)}}d\theta$$
 is the complete elliptic integral of the second kind.
 
When $k\to 0$, $K(k)$ has the expansion:
$$K(k)=\frac{\pi}{2}+\frac{\pi}{8}k^2+O(k^4)~,$$
whereas $E(k)$ has the expansion:
$$E(k)=\frac{\pi}{2}-\frac{\pi}{8}k^2+O(k^4)~,$$
This shows that
$$4(K(k)-E(k))= {\pi}k^2+O(k^4)~.$$
On the other hand,
$$\log K(\f_k)=\log\left({\frac{1+k^2}{1-k^2}}\right)=2k^2+O(k^6)~.$$
Therefore, using \eqref{eq double ineq integral},
$$\liminf_{k\to 0^+}\frac{\log K(\f_k)}{||\phi_k||_{cr}}\geq \lim_{k\to 0^+} \frac{2k^2+O(k^6)}{\pi k^2+O(k^4)}=\frac{2}{\pi}~.$$
Similarly, using the other inequality in \eqref{eq double ineq integral}, one shows:
$$\limsup_{k\to 0^+}\frac{\log K(\f_k)}{||\phi_k||_{cr}}\leq \lim_{k\to 0^+} \frac{2k^2+O(k^6)}{\pi k^2/\sqrt{1+k^2}+O(k^4)}=\lim_{k\to 0^+} \frac{2k^2+O(k^6)}{\pi k^2+O(k^4)}=\frac{2}{\pi}~,$$
which concludes the proof.
\end{proof}


\subsection*{Comparison when $k\to 1$}

Employing the same computations, we can study the limit of the ratio when $k$ approaches 1. 
\begin{reptheorem}{thm simple earth k=1}[Case $\lambda\to +\infty$]
Let $f_{\ph_\lambda}:\Hyp^2\to\Hyp^2$ be the minimal Lagrangian diffeomorphisms which extend a simple left earthquake $\ph_\lambda$ of weight $\lambda$. Then
$$\lim_{\lambda\to +\infty}\frac{\log K(f_{\ph_\lambda})}{||\ph_\lambda||_{cr}}=\frac{\sqrt 2}{2}\approx 0.71~.$$
\end{reptheorem}
\begin{proof}
To simplify the notation, we will prove the equivalent statement
$$\lim_{k\to 1^-}\frac{||\phi_k||_{cr}}{\log K(\f_k)}=\sqrt{2}~,$$
where $||\phi_k||_{cr}=\lambda_k$ is defined by the integral of Equation \eqref{eq lambda integral}.

As a preliminary remark, from the proof of the previous case let us define

$$F(k):=\int_{-\pi/2}^{\pi/2} \frac{2k^2\cos^2(\sigma)}{\sqrt{1-k^2\cos^2(\sigma)}}d\sigma=4(K(k)-E(k))~.$$
When $k\to 1$, 
$E(k)\to 1$, while $K(k)$ goes to infinity with the asymptotic expression (see \cite[Lemma 4]{MR1887769}):
$$K(k)=\log 4-\frac{1}{2}\log(1-k^2)+O(|(1-k^2)\log(1-k^2)|)~.$$
In particular, $\lambda_k\to+\infty$ as $k\to 1^-$. For the maximal dilatation, from Lemma \ref{lemma max dil simple},
$$\log K(\f_k)=\log\left(\frac{1+k^2}{1-k^2}\right)=\log(1+k^2)-\log(1-k^2)~.$$
Hence, when taking the limit for $k\to 1$, we obtain:

\begin{equation}\label{eq lime}
\lim_{k\to 1^-}\frac{F(k)}{\log K(\f_k)}= \lim_{k\to 1^-}\frac{4\left(\log 4-1-\frac{1}{2}\log(1-k^2)\right)}{\log(1+k^2)-\log(1-k^2)}={2}~.
\end{equation}

From Equation \eqref{eq lambda integral}, writing ${1-k^4\cos^4(u)}=(1-k^2\cos^2(u))(1+k^2\cos^2(u))$ and 
estimating $\sqrt{1+k^2\cos^2(u)}\leq \sqrt{2}$, we get (using \eqref{eq lime}):
$$\liminf_{k\to 1^-}\frac{||\phi_k||_{cr}}{\log K(\f_k)}\geq \frac{1}{\sqrt 2}\lim_{k\to 1^-}\frac{F(k)}{\log K(\f_k)}=\frac{2}{\sqrt 2}=\sqrt 2~.$$

To show the other inequality, fix $\epsilon>0$. Observe that $1+k^2\cos^2(u)\geq 2-\epsilon$ if and only if $k^2\cos^2(u)\geq 1-\epsilon$, which is equivalent to $|\sigma|\leq \sigma_0$ for
$$\sigma_0:=\arccos\left(\frac{\sqrt{1-\epsilon}}{k}\right)~.$$
Hence we have
$$\int_{-\sigma_0}^{\sigma_0} \frac{2k^2\cos^2(\sigma)}{\sqrt{1-k^4\cos^4(\sigma)}}d\sigma\leq \frac{1}{\sqrt{2-\epsilon}}\int_{-\sigma_0}^{\sigma_0} \frac{2k^2\cos^2(\sigma)}{\sqrt{1-k^2\cos^2(\sigma)}}d\sigma\leq \frac{1}{\sqrt{2-\epsilon}}F(k)~,$$
whereas for $|\sigma|\geq \sigma_0$, using $1+k^2\cos^2(\sigma)\geq 1$ and $k^2\cos^2(u)\leq 1-\epsilon$, we have
$$\int_{\sigma_0}^{\pi/2} \frac{2k^2\cos^2(\sigma)}{\sqrt{1-k^4\cos^4(\sigma)}}d\sigma\leq \frac{2}{\sqrt \epsilon}\int_{\sigma_0}^{\pi/2}d\sigma\leq \frac{\pi}{\sqrt \epsilon}~.$$
Thus we can estimate: 
$$
\frac{1}{\log K(\f_k)}\int_{-\pi/2}^{\pi/2} \frac{2k^2\cos^2(\sigma)}{\sqrt{1-k^4\cos^4(\sigma)}}d\sigma \\
\leq \frac{1}{\sqrt{2-\epsilon}}\frac{F(k)}{\log K(\f_k)}+\frac{2\pi}{\sqrt\epsilon\log K(\f_k)}~.
$$
Choosing $k$ large enough, the first term can be made smaller than $(2+\epsilon)/\sqrt{2-\epsilon}$ as a consequence of \eqref{eq lime}, while the second term can be male smaller than  $\epsilon$ (say) since $\log K(\f_k)$ goes to infinity. This shows that
 $$\limsup_{k\to 1^-}\frac{||\phi_k||_{cr}}{\log K(\f_k)}\leq \sqrt 2$$
and therefore concludes the proof.
\end{proof}


\section{Extending power functions} \label{sec power}

In this section, we will construct another one-parameter family of minimal Lagrangian maps $\mathfrak g_\alpha$, depending on a parameter $\alpha\in [1,+\infty)$. 

\subsection*{Coordinates from a geodesic}
It will be convenient here to use a different coordinate system for $\Hyp^2$. Let $\ell$ be a geodesic of $\Hyp^2$ as in the previous section, and suppose that in the upper half-plane model $\ell=\{x=0\}$, so that $T_\delta(z)=e^\delta\cdot z$ is the hyperbolic translation preserving $\ell$. Let $\gamma:\R\to\Hyp^2$ a unit speed parameterization of the geodesic orthogonal to $\ell$ through the point $i$, and consider the global coordinate $\varsigma:\R^2\to\Hyp^2$ defined by:
$$\varsigma(s,t)=T_t\circ \gamma(s)~.$$
It can be checked that in these coordinates the metric takes the form:
$$\varsigma^*g_{\Hyp^2}=ds^2+\cosh^2(s)dt^2~.$$
The hyperbolic translation takes the simple form
$$T_\delta(s,t)=(s,t+\delta)~.$$

\subsection*{Ansatz for self-adjoint Codazzi tensors}
In this section we will determine a bundle morphism $b\in\Gamma(\mathrm{End}(T\Hyp^2))$ satisfying the conditions of Proposition \ref{prop char min lag}, and we will later study the behavior of a minimal Lagrangian map associated. For this purpose, we will use the following property.

\begin{lemma}[{\cite{oliker},\cite[Lemma 2.1]{bonseppicodazzi}}] \label{lemma codazzi rho}
Given any smooth function $\rho:\Hyp^2\to\R$, the bundle morphism 
$$b=\Hess\rho-\rho\mathbbm 1$$
is self-adjoint for $g_{\Hyp^2}$ and satisfies the Codazzi equation $d^\nabla b=0$, 
where $\Hess\rho$ denotes the Hessian of $\rho$ for the metric $g_{\Hyp^2}$, as a (1,1)-tensor, and $\mathbbm 1$ is the identity operator.
\end{lemma}

In \cite{oliker} the converse statement was also proved, that is, any self-adjoint Codazzi tensor on a simply connected hyperbolic surface has the form $\Hess\rho-\rho\mathbbm 1$ for some function $\rho$.

In this section, using this approach, we will look for a function $\rho:\Hyp^2\to\R$ so that $b=\Hess\rho-\rho\mathbbm 1$ is a tensor as in Proposition \ref{prop char min lag}. Conversely to the previous section, here the condition $d^\nabla b=0$ will be automatically satisfied by Lemma \ref{lemma codazzi rho}, and the condition to impose is $\det b=1$.

The ansatz here is to consider a function which only depends on $s$, that is,
\begin{equation} \label{eq ansatz 2}
\rho(s,t)=r(s)~.
\end{equation}

\subsection*{Imposing the unit determinant condition}
Let us now compute the condition that $\det b=1$ for a tensor $b=\Hess\rho-\rho\mathbbm 1$.

\begin{lemma} \label{lemma ode unit det}
Let $\rho:\Hyp^2\to\R$ be a smooth function of the form
$$\rho(s,t)=r(s)$$
and let $b=\Hess\rho-\rho\mathbbm 1$. Then $\det b=1$ if and only 
\begin{equation} \label{eq det hess ode}
g'(s)=\tanh(s)\left(\frac{1}{g(s)}-g(s)\right)~,
\end{equation}
where $g(s)=r'(s)\tanh(s)-r(s)$ and $1/g(s)$ are the eigenvalues of $b$.
\end{lemma}
\begin{proof} We need to consider the condition
$$\det(\Hess\rho-\rho\mathbbm 1)=1$$
for the function $\rho(s,t)=r(s)$.
Recall that the hyperbolic Hessian is defined as
$$\Hess \rho(v)=\nabla_v\grad\rho~.$$
To compute the Hessian, consider the orthonormal frame
$$v_s=\partial_s\qquad v_t=\frac{1}{\cosh(s)}\partial_t~.$$
Observe that, by an explicit computation for the bracket,
$$\nabla_{v_s}v_t-\nabla_{v_t}v_s=[v_s,v_t]=-\tanh(s)v_t~,$$
hence $\nabla_{v_t}v_s=\tanh(s)v_t$ and $\nabla_{v_s}v_t=0$.

Since $\grad \rho(s,t)=r'(s)\partial_s=r'(s)v_s$, we have:
$$\Hess \rho(v_s)=\nabla_{v_s}(r'(s)v_s)=r''(s)v_s~,$$
whereas
$$\Hess \rho(v_t)=\nabla_{v_t}(r'(s)v_s)=r'(s)\nabla_{v_t}v_s=r'(s)\tanh(s)v_t~.$$
Hence in matrix form,
$$\Hess\rho-\rho\mathbbm 1=\begin{pmatrix} r''(s)-r(s) & 0 \\ 0 & r'(s)\tanh(s)-r(s) \end{pmatrix}~.$$

Let us define $g(s)=r'(s)\tanh(s)-r(s)$, so that $g(s)$ is an eigenvalue of $b$ at the point with coordinates $(s,t)$. Observe that 
$$g'(s)=r''(s)\tanh(s)+r'(s)(1-\tanh^2(s)-1)=\tanh(s)(r''(s)-r'(s)\tanh(s))~,$$
hence 
$$r''(s)-r(s)=\frac{1}{\tanh(s)}g'(s)+g(s)~.$$

This shows that $\Hess\rho-\rho\mathbbm 1$ takes the form:
$$\Hess\rho-\rho\mathbbm 1=\begin{pmatrix} \frac{1}{\tanh(s)}g'(s)+g(s) & 0 \\ 0 & g(s) \end{pmatrix}~,$$
and therefore we obtain the following ODE on $g$:
$$\det(\Hess\rho-\rho\mathbbm 1)=g(s)\left(\frac{1}{\tanh(s)}g'(s)+g(s)\right)=1~.$$
By algebraic manipulation (since, from the form of the equation, any solution $g$ cannot vanish at any point) this equation is equivalent to:
$$g'(s)=\tanh(s)\left(\frac{1}{g(s)}-g(s)\right)~,$$
as claimed.
\end{proof}

\subsection*{Solutions of ODE and definition of the diffeomorphisms $\g_\alpha$}

Let us now study the solutions of Equation \eqref{eq det hess ode}. We remark that $g(s)\equiv 1$ is a solution. Assuming $g(s)\neq 1$, we obtain the following form:
$$\frac{g'(s)g(s)}{1-g(s)^2}=\tanh(s)~.$$
Hence we can distinguish three cases:
\begin{itemize}
\item $g(s)=1$ is a solution for which $b=\mathbbm 1$.
\item If $g(s)<1$, then $1-g(s)^2>0$ and we obtain
$$-\frac{1}{2}\log (1-g(s)^2)=\log\cosh(s)+c~,$$
which leads to
$$1-g(s)^2=\frac{e^c}{\cosh^2(s)}=e^c(1-\tanh^2(s))~.$$
\item If $g(s)>1$, then $1-g(s)^2<0$ and we obtain
$$\frac{1}{2}\log (g(s)^2-1)=-\log\cosh(s)+c~,$$
which leads to 
$$g(s)^2-1=\frac{e^c}{\cosh^2(s)}=e^c(1-\tanh^2(s))~.$$
\end{itemize}

In conclusion, we have the solutions 
$$g_a(s)=\pm\sqrt{1+a(1-\tanh^2(s))}~,$$
where we can allow any value of $a\in\R$ and include all the three cases. We will only consider the positive root, since we can assume that $b$ is positive definite (and the other choice of sign corresponds to changing $b$ with $-b$).

Moreover, observe that $g_a$ is defined on $\R$ only if $a>-1$. Here we will only consider $a\geq 0$, so that $g_a$ is defined on $\R$ and $g(s)\geq 1$ --- namely, recalling Lemma \ref{lemma ode unit det}, $g(s)$ is the largest eigenvalue of $b$. To study the minimal Lagrangian diffeomorphisms associated to the tensors $b$ obtained in this way, we first need to analyze the metric $g_{\Hyp^2}(b\cdot,b\cdot)$.

\begin{lemma} \label{lemma global isometry}
Let $g_a:\R\to\R$, for $a\geq 0$, be defined by
$$g_a(s)=\sqrt{1+a(1-\tanh^2(s))}~,$$
and let $b_a$ be the self-adjoint, positive definite, Codazzi tensor
$$b_a=\begin{pmatrix} \frac{1}{g_a(s)} & 0 \\ 0 & g_a(s) \end{pmatrix}$$
in the $(s,t)$-coordinates of $\Hyp^2$. Then $g_{\Hyp^2}(b_a\cdot,b_a\cdot)$ has constant curvature $-1$ and is complete, hence it is isometric to $\Hyp^2$. Moreover, $\ell_0=\{s=0\}$ is a geodesic, while the $\{t=t_0\}$ are geodesics orthogonal to $\ell_0$ for every $t_0\in\R$.
\end{lemma}
\begin{proof}
First, we claim that $g_{\Hyp^2}(b_a\cdot,b_a\cdot)$ is a hyperbolic metric. In fact, since $d^\nabla b=0$, we have the following formula for the curvatures $\kappa$ (see \cite{labourieCP} or \cite{Schlenker-Krasnov}):
$$\kappa_{g_{\Hyp^2}(b_a\cdot,b_a\cdot)}=\frac{\kappa_{g_{\Hyp^2}}}{\det b_a}$$
and therefore $g_{\Hyp^2}(b_a\cdot,b_a\cdot)$ has curvature $-1$ since $g_{\Hyp^2}$ does, and $\det b_a=1$.

Now, in the $(s,t)$-coordinates, the metric $g_{\Hyp^2}(b_a\cdot,b_a\cdot)$ takes the form:
$$g_{\Hyp^2}(b_a\cdot,b_a\cdot)=\frac{1}{g_a(s)^2}ds^2+g_a(s)^2\cosh^2(s)dt^2~.$$
From the form of $g_a(s)$, it turns out that $g_a(-s)=g_a(s)$ and therefore $(s,t)\mapsto (-s,t)$ is an isometry for $g_{\Hyp^2}(b_a\cdot,b_a\cdot)$. This implies that $\ell_0=\{s=0\}$ is a geodesic. It is also complete, since the arclength parameter is $t'=t/g_a(0)$, which is defined in $\R$. Also the reflections $(s,t)\mapsto (s,t_0-t)$ are isometries for every $t_0\in\R$. Therefore the lines $\{t=t_0\}$ are geodesics for every $t_0$, and are clearly orthogonal to $\ell_0$. Finally, all such geodesics are complete, as a consequence of the fact that $g_a(s)$ is bounded from below by $1$ (recall that $a\geq 0$). This shows that $g_{\Hyp^2}(b_a\cdot,b_a\cdot)$ is isometric to $\Hyp^2$, and the isometries $(s,t)\mapsto (s,t+t_0)$ are hyperbolic isometries preserving $\ell_0$.
\end{proof}

In the following definition, we use the same notation as in Lemma \ref{lemma global isometry}.

\begin{defi}
For every $a\geq 0$, we define $\g_a:\Hyp^2\to\Hyp^2$ to be the isometry for the metrics $g_{\Hyp^2}(b_a\cdot,b_a\cdot)$ on the source and $g_{\Hyp^2}$ on the target, which (in the $(s,t)$-coordinates) maps the geodesic $\ell_0=\{s=0\}$ of $g_{\Hyp^2}(b_a\cdot,b_a\cdot)$  to the  geodesic $\ell=\{s=0\}$ of $g_{\Hyp^2}$.
\end{defi}

\subsection*{Computing the cross-ratio norm} Clearly, $\g_a$ is not an isometry, when one takes the metric $g_{\Hyp^2}$ both on the source and on the target. Up to composing with a hyperbolic translation preserving $\ell$, we can assume 
$$\g_a(0,0)=(0,0)~.$$
 In other words, $\g_a$ can be considered as the identity map of $\Hyp^2$, with respect to the metric $g_{\Hyp^2}$ on the source and  the metric $g_{\Hyp^2}(b_a\cdot,b_a\cdot)$ on the target. With this convention, $\g_a$ takes the following boundary value:

\begin{lemma} \label{lemma extension power}
The minimal Lagrangian diffeomorphisms $\g_a:\Hyp^2\to \Hyp^2$ extend to $\RP^1$, with boundary value (in the upper half-plane model) the homeomorphism: 
$$\psi_\alpha(x)=\begin{cases} x^{\alpha} & \textrm{if }x\geq 0 \\  -|x|^{\alpha} & \textrm{if }x< 0 \end{cases}~,$$
where $\alpha=\sqrt{1+a}$.
\end{lemma}
\begin{proof}
In the upper half-plane model, from Lemma \ref{lemma global isometry}, $\g_a$ preserves (setwise) the geodesic $\ell=\{x=0\}$, and maps geodesics orthogonal to $\ell$ to geodesics orthogonal to $\ell$. Moreover, observe that the coordinate $t$ is the arclength parameter of $\ell$ (which corresponds to points with coordinates $(0,t)$) for the metric $g_{\Hyp^2}$, while the arclength parameter for the metric $g_{\Hyp^2}(b_a\cdot,b_a\cdot)$ is $t'=t/g_a(0)$. That is, $\g_a$ stretches the geodesic $\ell$ by a uniform factor
$$g_a(0)=\sqrt{1+a}~.$$
Since moreover $\g_a$ maps the point $i$ (which has coordinates $(s,t)=(0,0)$) to itself, $\g_a$ sends a geodesic orthogonal to $\ell$, at (signed along $\ell$) distance $t$ from $i$, to the geodesic orthogonal to $\ell$ at signed distance $g_a(0)t$ from $i$.

Since the arclength parameter of $\ell$, in the upper half-plane model, is $t=\log y$, $\g_a$ maps $yi=e^t i$ to $e^{g_a(0)t}i=y^{\sqrt{1+a}}i$. But the endpoints at infinity of a geodesic orthogonal to $\ell$ through $yi\in\Hyp^2$ are $y$ and $-y$, and therefore $\g_a$ extends to $\RP^1$ with boundary value
$$\psi_a(x)=x^{\sqrt{1+a}}$$
if $x\geq 0$, and 
$$\psi_a(-x)=-(|x|^{\sqrt{1+a}})$$
if $x< 0$.
This concludes the proof.
\end{proof}

We conclude this subsection by giving an estimate on the cross-ratio norm of $\psi_\alpha$.

\begin{lemma} \label{lemma cross-ratio psi alpha}
Let $\psi_\alpha:\RP^1\to\RP^1$ be the orientation-preserving homeomorphism:
$$\psi_\alpha(x)=\begin{cases} x^{\alpha} & \textrm{if }x\geq 0 \\  -|x|^{\alpha} & \textrm{if }x< 0 \end{cases}~.$$
Then:
\begin{itemize}
\item $||\psi_\alpha||_{cr}\geq \log(2^\alpha-1)~,$
\item As $\alpha\to 1^+$, $||\psi_\alpha||_{cr}\geq C(\alpha-1)+O(\alpha-1)^2~,$ \\
where $C=\sqrt 2\left(\log (\sqrt 2+1)-\log (\sqrt 2-1)\right)\approx 2.49.$
\end{itemize}
\end{lemma}
\begin{proof}
For the first inequality, consider the quadruple $Q=(x_1,x_2,x_3,x_4)$ with:
$$x_1=0\qquad x_2=a\qquad x_3=2a\qquad x_4=\infty~,$$
(for some $a>0$), which is easily seen to be symmetric. Then 
$$cr(\psi_\alpha(Q))=-\frac{(2a)^\alpha-a^\alpha}{a^\alpha}=-(2^\alpha-1)~,$$
hence $|\log|cr(\psi_\alpha(Q))||=\log(2^\alpha-1)$. Since $||\psi_\alpha||_{cr}\geq |\log|cr(\psi_\alpha(Q))||$ by Definition \ref{defi cross ratio norm}, this concludes the first inequality.

To prove the second inequality, consider the following quadruple of points $Q'$:
$$x_1=-\sqrt 2-1\qquad x_2=1-\sqrt 2\qquad x_3=\sqrt 2-1\qquad x_4=\sqrt 2+1~.$$
By a direct computation using Equation \eqref{eq cross ratio}, we have $cr(Q')=-1$. On the other hand,
\begin{align*}
cr(\psi_\alpha(Q'))&=\frac{((\sqrt 2+1)^\alpha+(\sqrt 2+1)^\alpha)((\sqrt 2-1)^\alpha+(\sqrt 2-1)^\alpha)}{-((\sqrt 2+1)^\alpha-(\sqrt 2-1)^\alpha)^2} \\
&=-\frac{4(\sqrt 2+1)^\alpha(\sqrt 2-1)^\alpha}{((\sqrt 2+1)^\alpha-(\sqrt 2-1)^\alpha)^2}=-\frac{4}{((\sqrt 2+1)^\alpha-(\sqrt 2-1)^\alpha)^2}~.
\end{align*}
Now, observe that $(\sqrt 2+1)^\alpha-(\sqrt 2-1)^\alpha\geq 2$. Hence
\begin{align*}
|\log|cr(\psi_\alpha(Q'))||=&-\log\left(\frac{4}{((\sqrt 2+1)^\alpha-(\sqrt 2-1)^\alpha)^2}\right) \\
=&2\log\left((\sqrt 2+1)^\alpha-(\sqrt 2-1)^\alpha\right)-\log 4 \\
=&\left((\sqrt 2+1)\log (\sqrt 2+1)-(\sqrt 2-1)\log (\sqrt 2-1)\right)(\alpha-1)+O(\alpha-1)^2 \\
=&\sqrt 2\left(\log (\sqrt 2+1)-\log (\sqrt 2-1)\right)(\alpha-1)+O(\alpha-1)^2~.
\end{align*}
Since again $||\psi_\alpha||_{cr}\geq |\log|cr(\psi_\alpha(Q'))||$, this concludes the proof.
\end{proof}

\begin{remark} \label{remark optimal quadruple}
An asymptotic expansion of the function $\log(2^\alpha-1)$ which appears in the first inequality gives:
$$||\psi_\alpha||_{cr}\geq C(\alpha-1)+O(\alpha-1)^2~,$$
where one obtains $C=2\log 2\approx 1.39$. Hence the second choice of quadruple $Q$ in the proof of Lemma \ref{lemma cross-ratio psi alpha}
is necessary to refine the value of the constant $C$.

The second inequality could be achieved by taking any quadruple $Q'$, satisfying $cr(Q')=-1$, which is symmetric with respect to reflection in $x=0$. In fact, the homeomorphism $\psi_\alpha$ conjugates the action of the M\"obius transformation $x\mapsto \lambda x$ to the M\"obius transformation $x\mapsto \lambda^\alpha x$. Therefore all such quadruples $Q'$ are equivalent in terms of cross-ratio distortion of $\psi_\alpha$. Picking a quadruple of the form
$Q'=(-b,-a,a,b)$ for some $b>a>0$, a direct computation shows that $cr(Q')=-1$ if and only if 
$b=(3+2\sqrt 2)a=(\sqrt 2+1)^2 a$. This motivates the choice of the quadruple $Q'$ in the proof, which corresponds to the choice $a=\sqrt 2-1$ and thus $b=(\sqrt 2+1)^2(\sqrt 2-1)=\sqrt 2+1$.
\end{remark}

\subsection*{Computing the maximal dilatation}

On the other hand, we have the following value for the maximal dilatation of $\g_a$:

\begin{lemma} \label{lemma max dil power}
The maximal dilatation of the minimal Lagrangian diffeomorphisms $\g_a:\Hyp^2\to\Hyp^2$ is:
$$K(\g_a)=1+a~.$$
\end{lemma}
\begin{proof}
As for the proof of Lemma \ref{lemma max dil simple}, the eccentricity of the ellipses $d\g_a(\mathbb S^1)$ at every point can be computed as the ratio between the two eigenvalues of $b_a$, which are $g_a(s)$ and $1/g_a(s)$. Hence the eccentricity is:
$$g_a(s)^2=1+a(1-\tanh^2(s))~,$$
which is maximized for $s=0$. Hence we have
$$K(\g_a)=g_a(0)^2=1+a~,$$
as claimed.
\end{proof}

\subsection*{Comparison when $a\to 0$}
We have showed that the minimal Lagrangian diffeomorphism $\g_a$ coincides with the unique minimal Lagrangian extension $\g_a=f_{\psi_{\alpha}}$ of the power function $\psi_\alpha$, for $\alpha=\sqrt{1+a}$.
We are now ready to conclude the study of the ratio $\log K(f_{\psi_{\alpha}})/||\psi_\alpha||_{cr}$. Let us start by the case $\alpha\to 1$.

\begin{reptheorem}{thm power a=0}[Case $\alpha\to 1$]
Let $f_{\psi_\alpha}:\Hyp^2\to\Hyp^2$ be the minimal Lagrangian diffeomorphisms which extend the power function $\psi_\alpha$. Then
$$\limsup_{\alpha\to 1^+}\frac{\log K(f_{\psi_\alpha})}{||\psi_\alpha||_{cr}}\leq \frac{\sqrt 2}{\log \left(3+\sqrt 2\right)}~.$$
\end{reptheorem}
\begin{proof}
We shall study the limit of the ratio $\log K(f_{\psi_{\alpha}})/||\psi_\alpha||_{cr}$ as $\alpha\to 1^+$ --- that is, as $f_{\psi_{\alpha}}$ approaches the identity. On the one hand, we have from Lemma \ref{lemma cross-ratio psi alpha}:
$$||\psi_\alpha||_{cr}\geq \sqrt 2\left(\log (\sqrt 2+1)-\log (\sqrt 2-1)\right)(\alpha-1)+O(\alpha-1)^2~,$$
while on the other, recalling that $\alpha=\sqrt{1+a}$:
$$\log K(f_{\psi_\alpha})=\log(1+a)=\log\alpha^2=2\log\alpha=2(\alpha-1)+O((\alpha-1^2))~.$$
Therefore one has:
$$\limsup_{\alpha\to 1^+}\frac{\log K(f_{\psi_\alpha})}{||\psi_\alpha||_{cr}}\leq \frac{2}{\sqrt 2(\log (\sqrt 2+1)-\log (\sqrt 2-1))}=\frac{\sqrt 2}{\log \left(3+\sqrt 2\right)}\approx 0.80~,$$
as claimed.
\end{proof}

\begin{remark}
Numerical computation seems to show that $||\psi_\alpha||_{cr}$ is actually achieved at the quadruple $Q'$ introduced in the proof of Lemma \ref{lemma cross-ratio psi alpha}. The value ${\sqrt 2}/{\log \left(3+\sqrt 2\right)}\approx 0.80$ is thus expected to be the actual value of the limit  of ${\log K(f_{\psi_\alpha})}/{||\psi_\alpha||_{cr}}$ as $\alpha\to 1^+$. This would have as a consequence an improvement on the lower bound for the limit superior in Corollary \ref{cor ph to 0}.

However, computational difficulties are encountered when trying to prove that the second estimate of Lemma \ref{lemma cross-ratio psi alpha} is sharp, namely, that the quadruple $Q'$ achieves the cross-ratio norm (see also Remark \ref{remark optimal quadruple}), which is also an expected result given the symmetries of the problem. 

A nice way to produce an upper bound on the cross-ratio norm $||\psi_\alpha||_{cr}$ uses an argument similar to \cite[Proposition 2]{hu_muzician_douadyearle}. Namely, one can estimate $||\psi_\alpha||_{cr}$ in terms of the quantity $K_q(\psi_\alpha)$, which is defined as
$$K_q(\psi_\alpha)=\sup_Q\frac{M(\psi_\alpha(Q))}{M(Q)}~,$$
where $M$ denotes the conformal modulus of the quadrilateral with vertices in the quadruple $Q$. Following \cite{MR0344463}, one can then estimate $K_q$ by the maximal dilatation of the extremal extension, which by \cite{strebel1, strebel2} is known to be equal to $\alpha$. However, performing the precise computations (we omit the details here) leads to inequality:
$$\liminf_{\alpha\to 1^+}\frac{\log K(f_{\psi_\alpha})}{||\psi_\alpha||_{cr}}\geq 2\left(\frac{E(1/\sqrt 2)}{K(1/\sqrt 2)}\right)-1\approx 0.46~,$$
 which (unfortunately) does not improve the lower bound $1/2$ which we already know from Theorem \ref{thm seppi 2} (see Equation \eqref{eq seppi 2}). 

\end{remark}

\subsection*{Comparison when $a\to +\infty$}
Let us finally consider the case $a\to +\infty$.

\begin{reptheorem}{thm power a=0}[Case $\alpha\to +\infty$]
Let $f_{\psi_\alpha}:\Hyp^2\to\Hyp^2$ be the minimal Lagrangian diffeomorphisms which extend the power function $\psi_\alpha$. Then
$$\lim_{\alpha\to +\infty}\frac{\log K(f_{\psi_\alpha})}{||\psi_\alpha||_{cr}}=0~.$$
\end{reptheorem}
\begin{proof}
First, observe that $\alpha=\sqrt{1+a}$ tends to $+\infty$ as $a\to +\infty$, and $f_{\psi_\alpha}=\g_a$. 
From Lemma \ref{lemma cross-ratio psi alpha}, we have 
$$||\psi_\alpha||_{cr}\geq \log(2^\alpha-1)~.$$
As in the previous proof, from Lemma \ref{lemma max dil power} we have
$$\log K(f_{\psi_\alpha})=2\log \alpha~.$$
Hence we get 
$$\lim_{\alpha\to +\infty}\frac{\log K(f_{\psi_\alpha})}{||\psi_\alpha||_{cr}}\leq \lim_{\alpha\to +\infty}\frac{2\log \alpha}{\log(2^\alpha-1)}=0~.$$
This concludes the proof.
\end{proof}

\section{Conclusions} \label{sec conclusions}
As already observed in Section \ref{sec preliminaries}, see Theorem \ref{thm seppi} and Theorem \ref{thm seppi 2}, there exists a universal constant $C>0$ such that, for every quasisymmetric homeomorphism $\ph$,  the minimal Lagrangian extension $f_\ph$ satisfies:
$$\frac{\log K(f_{\ph})}{||\ph||_{cr}}\leq C~.$$
In this section we apply Theorem \ref{thm simple earth k=0} and Theorem \ref{thm power a=0} to give some conditions on the best possible constant $C$, in particular when $||\ph||_{cr}\to 0$ and $||\ph||_{cr}\to +\infty$.

Let us start by the case $||\ph||_{cr}\to 0$.

\begin{repcor}{cor ph to 0}
If $f_\ph:\Hyp^2\to\Hyp^2$ denotes the minimal Lagrangian extension of a quasisymmetric homeomorphism $\ph:\RP^1\to\RP^1$, then:
$$\liminf_{||\ph||_{cr}\to 0^+}\frac{\log K(f_\ph)}{||\ph||_{cr}}\in\left[\frac{1}{2},\frac{2}{\pi}\right]\qquad\text{and}\qquad \limsup_{||\ph||_{cr}\to 0^+}\frac{\log K(f_\ph)}{||\ph||_{cr}}\in\left[\frac{2}{\pi},+\infty\right)~.$$
\end{repcor}
\begin{proof}
The inequality
$$\liminf_{||\ph||_{cr}\to 0^+}\frac{\log K(f_\ph)}{||\ph||_{cr}}\geq \frac{1}{2}$$
follows from Theorem \ref{thm seppi 2}, see Equation \eqref{eq seppi 2}.

On the other hand, from Theorem \ref{thm simple earth k=0} the extensions $f_{\ph_\lambda}$ of the simple earthquake $\ph_\lambda$ satisfy:
$$\lim_{\alpha\to 1^+}\frac{\log K(f_{\psi_\alpha})}{||\psi_\alpha||_{cr}}= \frac{2}{\pi}~.$$
Hence this implies
$$\liminf_{||\ph||_{cr}\to 0^+}\frac{\log K(f_\ph)}{||\ph||_{cr}}\leq \frac{2}{\pi}~.$$

For the second claim, we know from Theorem \ref{thm seppi} that 
$$ \limsup_{||\ph||_{cr}\to 0^+}\frac{\log K(f_\ph)}{||\ph||_{cr}}$$
is finite. Finally, again from  Theorem \ref{thm simple earth k=0},
$$ \limsup_{||\ph||_{cr}\to 0^+}\frac{\log K(f_\ph)}{||\ph||_{cr}}\geq \frac{2}{\pi}~.$$
This concludes the proof of the corollary.
\end{proof}

When $||\ph||_{cr}\to +\infty$ we can conclude the following corollary.

\begin{repcor}{cor ph to infty}
If $f_\ph:\Hyp^2\to\Hyp^2$ denotes the minimal Lagrangian extension of a quasisymmetric homeomorphism $\ph:\mathbb S^1\to\mathbb S^1$, then:
$$\liminf_{||\ph||_{cr}\to +\infty}\frac{\log K(f_\ph)}{||\ph||_{cr}}=0\qquad\text{and}\qquad \limsup_{||\ph||_{cr}\to +\infty}\frac{\log K(f_\ph)}{||\ph||_{cr}}\in\left[\frac{\sqrt 2}{2},+\infty\right)~.$$
\end{repcor}
\begin{proof}
From Theorem \ref{thm power a=0}, for the extension $f_{\psi_\alpha}$ of the power function $\psi_\alpha$ we have 
$$\lim_{\alpha\to +\infty}\frac{\log K(f_{\psi_\alpha})}{||\psi_\alpha||_{cr}}=0~,$$
and therefore necessarily 
$$\liminf_{||\ph||_{cr}\to +\infty}\frac{\log K(f_\ph)}{||\ph||_{cr}}=0~.$$
On the other hand, Theorem \ref{thm simple earth k=0} shows
$$\lim_{\lambda\to +\infty}\frac{\log K(f_{\ph_\lambda})}{||\ph_\lambda||_{cr}}=\frac{\sqrt 2}{2}$$
for the extension $f_{\ph_\lambda}$ of the simple earthquake $\ph_\lambda$. 
Hence $$\limsup_{||\ph||_{cr}\to +\infty}\frac{\log K(f_\ph)}{||\ph||_{cr}}\geq\frac{\sqrt 2}{2}~.$$
Finally, finiteness of the limit superior is again consequence of Theorem \ref{thm seppi}.
\end{proof}

Let us also remark that, for the extension of the simple earthquake $\ph_\lambda$, Strebel in \cite{strebel1} and \cite{strebel2} proved that the extremal extension $f_{\ph_\lambda}^e$ of $\ph_\lambda$ satisfies $K(f_{\ph_\lambda}^e)=O(\lambda^2)$ as $\lambda\to +\infty$. On the other hand, from Theorem \ref{thm simple earth k=0} we showed that
$${K(f_{\ph_\lambda})}\geq {e^{C||\ph_\lambda||_{cr}}}$$
for large $\lambda$ and for some $C>0$. 
Hence we have:

\begin{repcor}{cor exp far}
The minimal Lagrangian extension of the simple earthquake $\ph_\lambda$ stays exponentially far from being extremal as $\lambda\to +\infty$.
\end{repcor}

Such result is the analogue for minimal Lagrangian extensions of \cite[Theorem 5]{hu_muzician_douadyearle}, which proves that the Douady-Earle extension of $\ph_\lambda$ stays exponentially far from being extremal as $\lambda\to +\infty$. More precisely, Hu and Muzician showed (see \cite[Theorem 3]{hu_muzician_douadyearle}) that, if $f_{\ph_\lambda}^{D\!E}$ denotes the Douady-Earle extension of $\ph_\lambda$, then
$$\lim_{\lambda\to +\infty}\frac{\log K(f^{D\!E}_{\ph_\lambda})}{||\ph_\lambda||_{cr}}\geq\frac{1}{4}~,$$
which should be compared with Theorem \ref{thm simple earth k=0} of the present paper.

\bibliographystyle{alpha}
\bibliographystyle{ieeetr}
\bibliography{../bs-bibliography}

\end{document}